\documentclass[reqno,centertags,12pt]{amsart}

\usepackage{amssymb,amsmath,amsthm,dsfont}
\usepackage[T1]{fontenc}

\theoremstyle{plain} \newtheorem{theorem}{Theorem}[section]
\newtheorem{lemma}[theorem]{Lemma}
\newtheorem{coro}[theorem]{Corollary}
\newtheorem{prop}[theorem]{Proposition}

\theoremstyle{definition} \newtheorem{definition}[theorem]{Definition}
\theoremstyle{remark} \newtheorem{remark}{Remark}

\renewcommand{\Re}{\operatorname{Re}}

\newcommand{\R}{{\mathbb R}}
\newcommand{\C}{{\mathbb C}}
\newcommand{\Z}{{\mathbb Z}}
\newcommand{\N}{{\mathbb N}}
\renewcommand{\S}{{\mathbb S}}
\newcommand{\T}{{\mathbb T}}

\newcommand{\sM}{{\mathcal M}}
\newcommand{\sS}{{\mathcal S}}

\newcommand{\sR}{{\mathcal R}}

\newcommand{\sN}{{\mathcal N}}
\newcommand{\eps}{\varepsilon}
\newcommand{\ls}{\lesssim}

\begin{document}

\title[The energy-critical NLS on a product of spheres]{The energy-critical nonlinear Schr\"odinger equation on a product of spheres}
\author[S.~Herr]{Sebastian~Herr}
\author[N.~Strunk]{Nils~Strunk}

\address{Universit\"at Bielefeld, Fakult\"at f\"ur Mathematik, Postfach 100131, 33501 Bielefeld, Germany}
\email{herr@math.uni-bielefeld.de}
\email{strunk@math.uni-bielefeld.de}

\begin{abstract}
Let $(M,g)$ be a compact smooth $3$-dimensional Riemannian manifold
without boundary. It is proved that the energy-critical nonlinear
Schr\"odinger equation is globally well-posed for small initial data
in $H^1(M)$, provided that a certain tri-linear estimate for free
solutions holds true. This estimate is known to hold true on the
sphere and tori in $3d$ and verified here in the case
$\S\times\S^2$. The necessity of a weak form of this tri-linear estimate is also discussed.
\end{abstract}
\keywords{Nonlinear Schr\"odinger Equation, compact manifold, well-posedness}
\subjclass[2010]{35Q55 (Primary); 35R01 (Secondary)}
\maketitle

\section{Introduction}\label{sect:intro}
Burq--G\'erard--Tzvetkov \cite{BGT04,BGT05a,BGT05b,BGT07} initiated a line of research on the well-posedness of nonlinear Schr\"odinger equations  on compact manifolds, extending Bourgain's results on tori \cite{B93a,B07}. More precisely, on a given compact smooth $d$-dimensional Riemannian manifold $(M,g)$ without boundary, the Cauchy-problem
\begin{equation}\label{eq:nls}
\left\{
\begin{aligned}
i\partial_t u+\Delta_gu&=\pm|u|^{p-1} u \\
u\big|_{t=0}&=u_0 \in H^s(M)
\end{aligned}
\right.
\end{equation}
is studied, where $u_0 \in H^s(M)$ is given initially and the aim is to prove the existence and uniqueness of a solution $u \in C([0,T),H^s(M,\C))$ and its continuous dependence  on $u_0$. For sufficiently smooth solutions $u$ the $L^2(M)$-norm and the energy
\[E(u)(t)=\frac12 \int_M|\nabla u(t,x)|^2 \,dx \pm \frac{1}{p+1} \int_M|u(t,x)|^{p+1} \,dx \]
are conserved quantities.
\medskip

On $M=\R^d$, solutions $u$ of the equation \eqref{eq:nls} can be rescaled to solutions $u_\lambda$ by setting
\[
u_\lambda(t,x)=\lambda^{\frac{2}{p-1}} u(\lambda^2 t, \lambda x) \qquad (\lambda > 0).
\]
The Sobolev semi-norm $\|\cdot\|_{\dot H^s(\R^d) }$ is invariant under this rescaling iff $s=s_c:=\frac{d}{2}-\frac{2}{p-1}$ and we call the range $s>s_c$ sub-critical, $s<s_c$ super-critical and $s=s_c$ critical.
In dimension $d=3$, the quintic problem ($p=5$) is called energy-critical since $s_c=1$. In this case, well-posedness in the critical space $H^1(M)$ is a key ingredient in the analysis of global well-posedness: For small initial data this immediately follows from the conservation of the energy $E(u)$, 
and in the defocusing case it serves as a starting point for a proof of global well-posedness for large initial data.

\medskip

Recently, first global results for \eqref{eq:nls} with $p=5$ in the critical space $H^1(M)$ have been obtained on the specific manifolds $M=\T^3$ \cite{HTT11,IP12,S14} and  $M=\S^3$ \cite{H13,PTW13} with standard metrics. These critical results crucially rely on precise spectral information. In this paper, we consider the manifold $M=\S\times\S^2$ with the standard metric. With regard to concentration of eigenfunctions and localization of the spectrum of $\Delta_g$ this is an intermediate case between $\mathbb{T}^3$ and $\mathbb{S}^3$, as explained in \cite[p.~257, l.~26ff]{BGT05b}. We consider this as a toy model for the central question concerning the critical well-posedness on arbitrary smooth compact Riemannian $3$-manifold, cp.\ \cite[p.~257, l.~31ff]{BGT05b}, as it forces us to unify some of the methods developed in \cite{HTT11,H13,HTT14}. On the other hand, its treatment requires new ideas, which we will point out below.

\medskip

Precisely, we focus on the following Cauchy-problem
\begin{equation}\label{eq:nls-sxs2}
\left\{
\begin{aligned}
i\partial_t u+\Delta_gu&=\pm|u|^{4} u\\
u\big|_{t=0}&=u_0\in H^s(\S\times\S^2)
\end{aligned}
\right.
\end{equation}
and we will prove the following in the critical case $s=1$:
\begin{theorem}\label{thm:main}
The Cauchy problem \eqref{eq:nls-sxs2} is globally well-posed for small initial data in $H^1(\S\times\S^2)$.
\end{theorem}
As usual, this result includes the existence of (mild) solutions $u
\in C(\R,H^1(\S\times\S^2))$, uniqueness in a certain subspace,
smooth dependence on the initial data and persistence of higher initial $H^s$-regularity. Our methods also imply local well-posedness for arbitrarily large initial data in $H^1(\S\times\S^2)$ by standard arguments, which we omit. We refer the reader to \cite[Theorem 1.1 and 1.2]{HTT11} for more explanations. In \cite{BGT05b} the global well-posedness in $H^1$ has been proved in the sub-quintic case (i.e.\ $1<p<5$), see \cite[Theorem 1]{BGT05b} for a more complete statement and \cite[Appendix A]{BGT05b} for an ill-posedness result in a super-quintic case.

\medskip

Generally speaking, the method of proof used here is similar to the
cases $M=\T^3$ \cite{HTT11} and $M=\S^3$ \cite{H13} and ideas from
\cite{BGT05b,BGT07} are used in order to deal with the fact that the
spectral cluster estimates are not optimal on $M=\S\times \S^2$, see
\cite[Theorem 3 and Remark 2.1]{BGT05b}. However, in the critical case
the tri-linear estimate obtained in \cite[Proposition~5.1]{BGT05b}
cannot be used because of the $\eps$-loss, which essentially comes
from the number-of-divisor-bound \cite[Lemma 4.2]{BGT05b}. The main
new estimate is a critical tri-linear estimate for free solutions, see
Proposition \ref{prop:trilin_est}, which is also known for $M=\T^3$
(both rational \cite[Proposition~3.5 and its proof, in particular
(26)]{HTT11} and irrational \cite[Proposition~4.1]{S14}) and $M=\S^3$
\cite[Proposition~3.6 and its proof, see (20)]{H13}. From this
estimate we derive the nonlinear estimate which is used for the Picard
iteration argument, which is along the lines of \cite{H13}.

\medskip

We point out that this reduction of the well-posedness proof to critical
tri-linear estimates for free solutions is independent of the specific
manifold. Conversely, we find that a weak form ($\delta=0$) of the estimate in
Proposition \ref{prop:trilin_est} is necessary for a well-posedness
result in $H^1(M)$ with a smooth flow map, which again does not depend
on the specific manifold $M=\S\times \S^2$.

\medskip

This paper is organized as follows: We conclude this section by
introducing some notation. In Section \ref{sect:multlin_est} we prove
the crucial tri-linear estimate for free solutions. In Section
\ref{sect:funct} we describe how the tri-linear estimate can be
extended to a certain function space, which allows us to perform the
standard Picard iteration argument. In Section \ref{sect:nec} we
discuss the necessity of a weak form of the tri-linear estimate for free solutions.

\subsection*{Acknowledgments}
The authors acknowledge support from the German Research Foundation,
Collaborative Research Center 701. The first author would like to
thank Nicolas Burq for raising a question which lead to Section \ref{sect:nec}, the second author would like thank Benoit Pausader for helpful remarks.

\subsection*{Notation}\label{subsect:not}
Let $(M,g)$ be a compact smooth $3$-dimensional Riemannian manifold without boundary. The spectrum
$\sigma(-\Delta_g)$ of the Laplace--Beltrami operator can be listed as
$0=\lambda_0 < \lambda_1 < \ldots < \lambda_n \to +\infty$. Let $h_k\colon L^2(M)\to L^2(M)$ be the
spectral projector onto the eigenspace corresponding to the eigenvalue $\lambda_k$. For $f\in L^2(M)$ and a
dyadic number $N\in\N$, we define the projector 
\[
   P_Nf = \sum_{k\in\N_0:\;N\leq\langle\lambda_k\rangle^{\frac12}<2N} h_k(f),
\]
where $\langle x\rangle = (1+|x|^2)^{\frac12}$.
We define the usual $L^2$-based Sobolev space $H^s(M)=(1-\Delta_g)^{\frac{s}{2}}L^2(M)$, equipped with the norm
\[
   \|f\|_{H^s(M)} = \biggl(\sum_{k\in\N_0} \langle \lambda_k\rangle^{s} \|h_k(f)\|_{L^2(M)}^2\biggr)^{\frac12}.
\]
Due to $L^2$-orthogonality we have
\[
   \|f\|_{H^s(M)}^2\sim \sum_{N \geq 1} N^{2s}\|P_N f\|_{L^2(M)}^2.
\]
Here and in the sequel $\sum_{N\geq 1}$ indicates that we are summing over all $N=1,2,4,8,\ldots$.

In the case $M=\S\times\S^2$ we use the same notation for the spectrum and the spectral projectors as in
\cite[Section~5]{BGT05b}: The spectrum of $-\Delta=-\Delta_g$ is given by
\[
   \lambda_{m,n}=m^2+n^2+n,\quad (m,n)\in\Z\times \N_0.
\]
We denote by $\Pi_n\colon L^2(\S^2)\to L^2(\S^2)$ the spectral projector onto spherical harmonics of degree $n$ on $\S^2$. For functions $f$ on $M$ we write $\S\times\S^2\ni (\theta,\omega)\mapsto f(\theta,\omega)$. The $m$-th Fourier-coefficient of $f(\,\cdot\,,\omega)$ is defined by
\[
   \Theta_mf(\omega):=\frac{1}{2\pi} \int_0^{2\pi} f(\theta,\omega)e^{-im\theta} \,d\theta.
\]
Hence, for $f \in L^2(M)$, we have
\[
   f(\theta,\omega)=\sum_{(m,n)\in\Z\times \N_0}e^{im\theta}\Pi_n \Theta_m(f)(\omega)
\]
in the $L^2$-sense.
For dyadic $N$ we define the projector
\[
   P_Nf(\theta,\omega)
   = \sum_{\substack{(m,n)\in\Z\times \N_0:\\N\leq\langle\lambda_{m,n}\rangle^{\frac12}< 2N}}e^{im\theta}\Pi_n \Theta_m(f)(\omega).
\]
We define the Sobolev space $H^s(M)=(1-\Delta_g)^{\frac{s}{2}}L^2(M)$, equipped with the norm
\[
  \|f\|_{H^s(M)}^2 = \sum_{(m,n)\in\Z\times\N_0} \langle \lambda_{m,n}\rangle^{s}
  \|\Pi_n\Theta_m f\|_{L^2(M)}^2
  \sim \sum_{N \geq 1} N^{2s}\|P_N f\|_{L^2(M)}^2.
\]

\section{The tri-linear estimate for free solutions}\label{sect:multlin_est}

In this section we are going to prove a new tri-linear Strichartz estimate for free solutions
(Proposition~\ref{prop:trilin_est}). This proposition is an improvement of the tri-linear estimate
\cite[Proposition~5.1]{BGT05b} of Burq--G\'erard--Tzvetkov in the sense that it is critical.

\medskip

We start this section collecting two known results, which we will rely on later.
The following estimate on exponential sums is due to Bourgain \cite{B93a} and was used to prove
Strichartz estimates on the flat torus.
\begin{lemma}[cp.\ {\cite[Formula~(3.116)]{B93a}}]\label{lem:exp_sum_est}
  Let $p>4$, then for all $N\geq1$, $a\in\ell^2(\Z^2)$, $z\in\Z^2$ and $\sS_N\subseteq z + \{-N,\ldots,N\}^2$
  it holds that
\[
   \biggl\|\sum_{n\in\sS_N} e^{-i|n|^2t} e^{in\cdot x} a_n\biggr\|_{L^p_{t,x}([0,2\pi]^3)} \ls N^{1-\frac4p} \|a\|_{\ell^2}.
\]
\end{lemma}
\begin{proof}[Proof/Reference]
  The desired estimate follows immediately from the Galilean transformation
\[
   x\cdot(n-z) - t|n-z|^2 = (x+2tz)\cdot n - t|n|^2 - x\cdot z - t|z|^2,
\]
  as applied in \cite[formulas (5.7)--(5.8)]{B93a} and \cite[Proposition~3.1]{HTT11},
  and from \cite[formula~(3.116)]{B93a}.
\end{proof}

We will also use the succeeding tri-linear spectral cluster estimate of Burq--G\'erard--Tzvetkov, which is more
generally valid for any compact smooth Riemannian manifold without boundary of dimension two.
\begin{lemma}[{\cite[Theorem~3]{BGT05b}}]\label{lem:trilin_spect_clust_est}
  For all integers $n_1\geq n_2\geq n_2\geq 0$ and $f_1,f_2,f_3\in L^2(\S^2)$ the following
  tri-linear estimate holds true
\[
   \|\Pi_{n_1}f_1\Pi_{n_2}f_2\Pi_{n_3}f_3\|_{L^2(\S^2)}
   \lesssim \bigl(\langle n_2\rangle\langle n_3\rangle\bigr)^{\frac14}\prod_{j=1}^3\|\Pi_{n_j}f_j\|_{L^2(\S^2)}.
\]
\end{lemma}

Throughout this paper, let $\tau_0=[0,8\pi]$ be the considered time interval.
For the purpose of proving Proposition~\ref{prop:trilin_est},
we will use following exponential sum estimate.
The main idea is to reduce the estimate to Lemma~\ref{lem:exp_sum_est}.
\begin{lemma}\label{lem:lin_est}
  Let $p>4$. Then, for all $N\geq1$, $a\in\ell^2(\Z^2)$, $z\in\Z^2$ and
  $\sS_N\subseteq z+\{-N,\ldots,N\}^2$ the estimate
\[
   \biggl\|\sum_{(m,n)\in \sS_N} e^{-i\lambda_{m,n}t}e^{im\theta} a_{m,n}\biggr\|_{L_{t,\theta}^p(\tau_0\times\S)}
   \ls N^{1-\frac3p} \|a\|_{\ell^2}
\]
holds true.
\end{lemma}
\begin{proof}
  We first show that we may replace $\lambda_{m,n}$ by $m^2+n^2$.
  We set $4\widetilde t = t$ and $2\widetilde\theta=\theta$.
  Since $4\lambda_{m,n}=(2m)^2+(2n+1)^2-1$, the left hand side is
  bounded by a constant times
\[
   \biggl\|\sum_{(\widetilde m,\widetilde n)\in\widetilde{\sS_N}} e^{-i(\widetilde m^2+\widetilde n^2)\widetilde t}
      e^{i\widetilde m\widetilde\theta} \widetilde a_{\widetilde m,\widetilde n}\biggr\|_{L^p_{\widetilde t,\widetilde\theta}([0,2\pi]^2)},
\]
  where $\widetilde{\sS_N} := \{(\widetilde m,\widetilde n)\in\Z^2 : (\widetilde m/2,(\widetilde n-1)/2)\in\sS_N\}$
  is inside a cube of side length $4N$, and
\[
   \widetilde a_{\widetilde m,\widetilde n} :=
   \begin{cases}
     a_{\widetilde m/2,(\widetilde n-1)/2}, & \widetilde m\in2\Z,\; \widetilde n\in2\Z+1,\\
     0, & \text{otherwise.}
   \end{cases}
\]
  Hence, it suffices to prove
\[
   \biggl\|\sum_{(m,n)\in\sS_N} e^{-i(m^2+n^2)t} e^{im\theta} a_{m,n}\biggr\|_{L^p_{t,\theta}([0,2\pi]^2)}
   \lesssim  N^{1-\frac3p} \|a\|_{\ell^2}.
\]
  In order to apply the exponential sum estimate of Lemma~\ref{lem:exp_sum_est}, we introduce another
  variable $\nu$. Obviously, the left hand side is bounded by
\[
   \sup_{\nu\in[0,2\pi]}\biggl\|\sum_{(m,n)\in\sS_N} e^{-i(m^2+n^2)t} e^{im\theta}e^{in\nu} a_{m,n}\biggr\|_{L^p_{t,\theta}([0,2\pi]^2)},
\]
  which can be further estimated by
\[
   \biggl\|\sum_{(m,n)\in\sS_N} e^{-i(m^2+n^2)t} e^{im\theta}e^{in\nu} a_{m,n}\biggr\|_{L^p_{t,\theta}([0,2\pi]^2,L^\infty_\nu([0,2\pi]))}
\]
  using Minkowski's inequality. Sobolev's embedding in $\nu$ allows to bound this by a constant times
\[
   N^{\frac1p}\biggl\|\sum_{(m,n)\in\sS_N} e^{-i(m^2+n^2)t} e^{im\theta}e^{in\nu} a_{m,n}\biggr\|_{L^p_{t,\theta,\nu}([0,2\pi]^3)}.
\]
  Finally, Lemma~\ref{lem:exp_sum_est} implies the desired result.
\end{proof}

\begin{remark}
  One can even lower the exponent w.r.t.\ $\S$ to $4$, if the exponent w.r.t.\ time
  is raised to $p>\frac{16}{3}$: Let $p>\frac{16}{3}$, then, under the same assumptions on
  $a$, $N$, $\sS_N$ as in  Lemma~\ref{lem:lin_est}, the following estimate holds true:
\begin{equation}\label{eq:esp_sum_est_tori_method}
   \biggl\|\sum_{(m,n)\in \sS_N} e^{-i\lambda_{m,n}t}e^{im\theta} a_{m,n}\biggr\|_{L_t^p(\tau_0,L_\theta^4(\S))}
   \lesssim  N^{\frac34-\frac2p} \|a\|_{\ell^2}.
\end{equation}
  The proof is very similar to Bourgain's proof of Strichartz estimates on 
  irrational tori \cite[Proposition~1.1]{B07}.
  However, it seems that this estimate is not appropriate for studying local existence:
  We start with a tri-linear $L^2(\tau_0\times M)$ estimate and proceed as in the proof of
  Proposition~\ref{prop:trilin_est} until \eqref{eq:trilin_first_step}.
  Then, using H\"older's inequality to put the two functions with the highest
  frequencies to $L^{\frac{16}{3}+}_tL^4_\theta$ and thus the function with the lowest frequency, say $N_3$,
  to $L^{8-}_tL^{\infty}_\theta$. We treat the latter term as follows:
  Applying Sobolev's embedding to bound it by the $L^{8-}_tL^{4}_\theta$-norm gives a factor $N_3^{\frac14}$.
  The exponential sum estimate \eqref{eq:esp_sum_est_tori_method} gives
  $N_3^{\frac12-}$ and from the spectral cluster estimate we get another $N_3^{\frac14}$ as in
  \eqref{eq:trilin_first_step}.
  All in all we obtain $N_3^{1-}$, and hence the power on the lowest frequency is too low
  to conclude local well-posedness. 
\end{remark}

  The subsequent estimate will serve as an $L^\infty(\tau_0\times \S)$ estimate. It improves the  previous
  lemma, because it takes additional smallness properties of the underlying point set $\sS_{N,M}$ into account, which
  will be induced by almost orthogonality in time. 
\begin{lemma}\label{lem:lin_infty_est}
  Let $a\in\ell^2(\Z^2)$, $N\geq M\geq 1$, and
\[
  \sS_{N,M}\subseteq \bigl\{ (m,n)\in z+\{0,\ldots,N\}^2 : \sqrt{\lambda_{m,n}} \in [b,b+M] \bigr\}
\]
  for some $z\in\Z^2$ and $b\in\N_0$. Then we have
\[
  \sum_{(m,n)\in \sS_{N,M}}|a_{m,n}| \lesssim M^{\frac12} N^{\frac12} \|a\|_{\ell^2}.
\]
\end{lemma}
\begin{proof}
  By Cauchy-Schwarz, we only have to show $\#\sS_{N,M}\ls MN$.
  Since
\[
   \#\sS_{N,M}
   \leq \#\bigl\{ (m,n)\in \widetilde z+\{0,\ldots,2N\}^2 : \sqrt{m^2+n^2} \in [2b,2b+4M] \bigr\},
\]
  where $\widetilde z=2z+(0,1)$, we may assume $\lambda_{m,n}=m^2+n^2$.
  The rest of the proof is motivated by \cite[Section~2.7]{G85}.
  Consider all the lattice points in $\sS_{N,M}$ as centers of unit squares with sides parallel to the coordinate
  axes. Obviously, the number of lattice points in $\sS_{N,M}$ equals the area of the union of these squares.
  The diagonal of the unit squares is $\sqrt2$. Consequently, the union of the squares is inside
  a $\frac{1}{\sqrt2}$-neighborhood of $\sS_{N,M}$. This neighborhood can be covered by an annulus of angle
  $\alpha$, outer radius $R:=2b+5M$ and inner radius $r:=\max\{R-6M,0\}$, where $\alpha\in[0,2\pi]$ is determined
  as follows: Since the point set is located in a cube of size $N$, the arc length of the annulus sector is
  bounded by $\sim N$. Thus $\alpha\sim\frac{N}{R}$, and we deduce that the area is bounded by
\[
    \frac{\alpha}{2}(R^2-r^2) \ls \frac{N}{R}MR \ls MN.
\]
\end{proof}

  Interpolating Lemma~\ref{lem:lin_est} and Lemma~\ref{lem:lin_infty_est}, we obtain an
  $L^p(\tau_0\times\S)$ estimate for $p>4$ that takes additional smallness
  properties of the underlying point set into account as Lemma~\ref{lem:lin_infty_est} does.
\begin{coro}\label{coro:lin_est2}
  Let $p>4$. Then, for all $\varepsilon>0$, $\sS_{N,M}$ as in Lemma~\ref{lem:lin_infty_est},
  $N\geq M\geq 1$ and $a\in\ell^2(\Z^2)$ we have that
\[
  \biggl\| \sum_{(m,n)\in\sS_{N,M}} e^{-i\lambda_{m,n}t}e^{im\theta} a_{m,n}
  \biggr\|_{L^p_{t,\theta}(\tau_0\times\S)}
  \lesssim \Bigl(\frac{N}{M}\Bigr)^\varepsilon N^{\frac12-\frac1p} M^{\frac12-\frac2p} \|a\|_{\ell^2}.
\]
\end{coro}
 \begin{proof}
   We set $f(t,\theta):=\bigl|\sum_{(m,n)\in\sS_{N,M}} e^{-i\lambda_{m,n}t}e^{im\theta} a_{m,n}\bigr|$ for brevity.
   The estimate is nontrivial only if $\varepsilon\leq\frac{p-4}{2p}$. Furthermore, we set
   $\varepsilon'=2p\varepsilon>0$  and $\vartheta=\frac{4+\varepsilon'}{p}\leq 1$. Then,
   H\"older's inequality, Lemma~\ref{lem:lin_est} and Lemma~\ref{lem:lin_infty_est} imply
\[
   \|f\|_{L^p_{t,\theta}} = \|f^\vartheta f^{1-\vartheta}\|_{L^p_{t,\theta}}
   \leq \|f\|_{L^{4+\varepsilon'}_{t,\theta}}^\vartheta \|f\|_{L^\infty_{t,\theta}}^{1-\vartheta}
   \lesssim \Bigl(\frac{N}{M}\Bigr)^\varepsilon N^{\frac12-\frac1p} M^{\frac12-\frac2p} \|a\|_{\ell^2}.
\]
\end{proof}

\begin{prop}\label{prop:trilin_est}
  There exists $\delta>0$ such that for all $\phi_1,\phi_2,\phi_3\in L^2(M)$ and dyadic numbers
  $N_1\geq N_2\geq N_3\geq 1$ the estimate
\begin{multline*}
   \|P_{N_1}e^{it\Delta}\phi_1P_{N_2}e^{it\Delta}\phi_2P_{N_3}e^{it\Delta}\phi_3\|_{L^2(\tau_0\times M)}\\
   \lesssim \biggl(\frac{N_3}{N_1}+\frac{1}{N_2}\biggr)^\delta
   N_2 N_3 \prod_{j=1}^3 \|\phi_j\|_{L^2(M)}
\end{multline*}
  holds true.
\end{prop}
\begin{proof}
  We will exploit almost orthogonality in the first three steps to show that we may assume the
  highest frequency to be further localized. In the last step we will estimate the remaining tri-linear
  estimate using the foregoing results.
  First, we recall that for $t\in\tau_0$ and $(\theta,\omega)\in\S\times\S^2$
\begin{multline*}
   P_{N_1}e^{it\Delta}\phi_1P_{N_2}e^{it\Delta}\phi_2P_{N_3}e^{it\Delta}\phi_3(\theta,\omega)\\
   = \sum_{\sN} \prod_{j=1}^3 e^{-i\lambda_{m_j,n_j}t} e^{im_j\theta} \Pi_{n_j}\Theta_{m_j}\phi_j(\omega),
\end{multline*}
  where $\sN=\sN_1\times\sN_2\times\sN_3$ and
\begin{equation}\label{eq:def_set}
   \sN_j=\bigl\{ (m,n)\in\Z\times\N_0 : N_j\leq\langle\lambda_{m,n}\rangle^{\frac12}<2N_j \bigr\},
   \quad j=1,2,3.
\end{equation}
  In this proof $\sum_{\sN}$ should be understood as $\sum_{(m_1,n_1,m_2,n_2,m_3,n_3)\in\sN}$.

  We apply step~1--3 only if $N_1>N_2$, otherwise we will proceed with step~4 directly
  (with $\sS:=\{N_1,\ldots,2N_1-1\}$ and $M:=N_1=N_2$).

  \emph{Step 1.}
  Due to spatial almost orthogonality induced by the $\S$ component, it suffices to prove the desired estimate
  in the case
\[
   P_{\sR}P_{N_1}e^{it\Delta}\phi_1P_{N_2}e^{it\Delta}\phi_2P_{N_3}e^{it\Delta}\phi_3,
\]
  where $\sR\subseteq [b,b+N_2]\times[0,2N_1]$ for some $b\in\Z$. We spell out more details in the next step.

  \emph{Step 2.}
  Now, we use almost orthogonality that comes from the $\S^2$ component. It is a well-known
  fact that the product
  of a spherical harmonic of degree $n$ with another of degree $m$ can be expanded in terms of spherical
  harmonics of degree less or equal to $n+m$. Furthermore, it is well-known that two spherical harmonics of
  different degree are orthogonal in $L^2(\S^d)$, $d\in\N$.
  We finally remark that complex conjugation does not change the degree of a spherical
  harmonic. Details may be found in \cite[Section~VI.2]{SW71}. Now, we prove that it suffices
  to consider the case, where $n_1$ is located in an interval of the size of the second
  highest frequency $N_2$. To that purpose, we define the following partition of $\N_0$:
\[
   \N_0=\dot{\bigcup_{k\in\N_0}} I_k,\quad \text{where}\quad I_k=\bigl[kN_2,(k+1)N_2\bigr).
\]
  We claim that for fixed $\theta\in\S$ and $t\in\tau_0$ it holds that
\begin{multline*}
   \|P_{\sR}P_{N_1}e^{it\Delta}\phi_1P_{N_2}e^{it\Delta}\phi_2P_{N_3}e^{it\Delta}\phi_3(\theta)\|_{L^2(\S^2)}^2\\
   \sim \sum_{k\in\N_0} \bigl\|P_{\sR_k}P_{N_1}e^{it\Delta}\phi_1
        P_{N_2}e^{it\Delta}\phi_2P_{N_3}e^{it\Delta}\phi_3(\theta)\bigr\|_{L^2(\S^2)}^2,
\end{multline*}
  where $\sR_k=\sR\cap(\Z\times I_k)$. Let $k,\widetilde k\in\N_0$, then
\begin{multline*}
   \bigl\langle P_{\sR_k}P_{N_1}e^{it\Delta}\phi_1P_{N_2}e^{it\Delta}\phi_2P_{N_3}e^{it\Delta}\phi_3(\theta), \\
      P_{\sR_{\widetilde k}}P_{N_1}e^{it\Delta}\phi_1P_{N_2}e^{it\Delta}\phi_2
      P_{N_3}e^{it\Delta}\phi_3(\theta)\bigr\rangle_{L^2(\S^2)} \\
   = \sum_{\substack{\sR_k\times\sN_2\times\sN_3,\\\sR_{\widetilde k}\times\sN_2\times\sN_3}}
      I_{{\bf m}, {\bf n}}\prod_{j=1}^3 e^{-i(\lambda_{m_j,n_j}-\lambda_{\widetilde m_j,\widetilde n_j})t}
      e^{i(m_j-\widetilde m_j)\theta},
\end{multline*}
  where ${\bf m}=(m_1,m_2,m_3,\widetilde m_1, \widetilde m_2, \widetilde m_3)$,
  ${\bf n}=(n_1,n_2,n_3,\widetilde n_1,\widetilde n_2,\widetilde n_3)$, and
  $I_{{\bf m}, {\bf n}}$ is defined by
\[
   I_{{\bf m}, {\bf n}} = \int_{\S^2} \prod_{j=1}^3 \Pi_{n_j}\Theta_{m_j}\phi_j(\omega)
      \overline{\Pi_{\widetilde n_j}\Theta_{\widetilde m_j}\phi_j(\omega)} \; d\omega.
\]
  Without loss of generality we may assume $n_1>\widetilde n_1$. Then
\[
   Y:=\overline{\Pi_{\widetilde n_1}\Theta_{\widetilde m_1}\phi_j}
   \prod_{j=2}^3 \Pi_{n_j}\Theta_{m_j}\phi_j
      \overline{\Pi_{\widetilde n_j}\Theta_{\widetilde m_j}\phi_j}\in L^2(\S^2)
\]
  can be expanded in terms of spherical harmonics of degree less or
  equal to $\widetilde n_1+8N_2$. Hence, if $|k-\widetilde k|\gg 1$, then
\[
   I_{{\bf m}, {\bf n}} = \bigl\langle \Pi_{n_1}\Theta_{m_1}\phi_1, 
     \overline{Y} \bigr\rangle_{L^2(\S^2)} = 0.
\]

  \emph{Step 3.} Using almost orthogonality in time, we may gain a small power of
    $M:=\max\bigl\{\frac{N_2^2}{N_1},1\bigr\}$. Similar ideas have been used in the proofs of
    \cite[Proposition~3.5]{HTT11} and \cite[Proposition~3.6]{H13}, for instance. We define the partition
\[
   \N_0=\dot{\bigcup_{\ell\in\N_0}} J_\ell\quad \text{where}\quad
   J_\ell=\bigl[\ell M,(\ell+1)M\bigr).
\]
   We show that we may assume $\sqrt{\lambda_{m_1,n_1}}$ to vary in an interval of length
   $M$: Fix $(\theta,\omega)\in \S\times\S^2$ and set
\[
   \sS_{k,\ell}=\bigl\{(m_1,n_1)\in \sR_k: \sqrt{\lambda_{m_1,n_1}}\in J_\ell \bigr\}, \quad k,\ell\in\N_0,
\]
   then we claim that
\begin{multline*}
   \|P_{\sR}P_{N_1}e^{it\Delta}\phi_1P_{N_2}e^{it\Delta}\phi_2
          P_{N_3}e^{it\Delta}\phi_3(\theta,\omega)\|_{L^2_t(\tau_0)}^2\\
   \sim \sum_{k,\ell\in\N_0} \|P_{\sS_{k,\ell}}P_{N_1}e^{it\Delta}\phi_1P_{N_2}e^{it\Delta}\phi_2
          P_{N_3}e^{it\Delta}\phi_3(\theta,\omega)\bigr\|_{L^2_t(\tau_0)}^2.
\end{multline*}
   We consider the inner product
\begin{multline*}
   \bigl\langle
     P_{\sS_{k,\ell}}P_{N_1}e^{it\Delta}\phi_1P_{N_2}e^{it\Delta}\phi_2P_{N_3}e^{it\Delta}\phi_3(\theta,\omega),\\
     P_{\sS_{k,\widetilde \ell}}P_{N_1}e^{it\Delta}\phi_1P_{N_2}e^{it\Delta}\phi_2P_{N_3}e^{it\Delta}\phi_3(\theta,\omega)
   \bigr\rangle_{L^2_t(\tau_0)}\\
   = \sum_{\substack{\sS_{k,\ell}\times\sN_2\times\sN_3,\\\sS_{k,\widetilde\ell}\times\sN_2\times\sN_3}}
       I_{{\bf m}, {\bf n}} \prod_{j=1}^3 e^{i(m_j-\widetilde m_j)\theta}
       \Pi_{n_j}\Theta_{m_j}\phi_j(\omega)\overline{\Pi_{\widetilde n_j}\Theta_{\widetilde m_j}\phi_j(\omega)},
\end{multline*}
   where ${\bf m}=(m_1,m_2,m_3,\widetilde m_1, \widetilde m_2, \widetilde m_3)$,
   ${\bf n}=(n_1,n_2,n_3,\widetilde n_1,\widetilde n_2,\widetilde n_3)$, and
\[
   I_{{\bf m}, {\bf n}} = \int_{\tau_0} e^{-i(\lambda_{m_1,n_1}+\lambda_{m_2,n_2}+\lambda_{m_3,n_3}
      -\lambda_{\widetilde m_1,\widetilde n_1}-\lambda_{\widetilde m_2,\widetilde n_2}-\lambda_{\widetilde m_3,\widetilde n_3})t}\,dt.
\]
   Assuming $|\ell-\widetilde\ell|\gg 1$, we may estimate the modulus of the phase from below by
\[
   \bigl|\bigl(\sqrt{\lambda_{m_1,n_1}}+\sqrt{\lambda_{\widetilde m_1,\widetilde n_1}}\bigr)
   \bigl(\sqrt{\lambda_{m_1,n_1}}-\sqrt{\lambda_{\widetilde m_1,\widetilde n_1}}\bigr)\bigr|-16N_2^2\\
   \gtrsim |\ell-\widetilde\ell|N_2^2,
\]
   and since all the eigenvalues are integers, we deduce $ I_{{\bf m}, {\bf n}} =0$.

  \emph{Step 4.} Thanks to the first three steps, we may replace $P_{N_1}e^{it\Delta}\phi_1$ by
  $P_\sS P_{N_1}e^{it\Delta}\phi_1$, where $\sS=\sS_{k,\ell}$ for some $k,\ell\in N_0$.
  Recall that for $t\in\tau_0$ and $(\theta,\omega)\in\S\times\S^2$
\begin{multline*}
   P_{\sS}P_{N_1}e^{it\Delta}\phi_1P_{N_2}e^{it\Delta}\phi_2P_{N_3}e^{it\Delta}\phi_3(\theta,\omega)\\
   = \sum_{\sM} \prod_{j=1}^3 e^{-i\lambda_{m_j,n_j}t} e^{im_j\theta} \Pi_{n_j}\Theta_{m_j}\phi_j(\omega),
\end{multline*}
  where $\sM:=\sS\times\sN_2\times\sN_3$ and $\sN_j$, $j=2,3$, are defined
  in \eqref{eq:def_set}.
  The next step is a nice way to treat the $L^2(\S^2)$-norm separately without losing
  oscillations in the $\S$ component and in time. Note that this was also used by
  Burq--G\'erard--Tzvetkov in the proof of \cite[Proposition~5.1]{BGT05b}.
  Plancherel's identity with respect to $t$ and $\theta$ and the triangle inequality for the
  $L^2(\S^2)$ norm yield
\begin{multline*}
  \| P_{\sS}P_{N_1}e^{it\Delta}\phi_1P_{N_2}e^{it\Delta}\phi_2P_{N_3}e^{it\Delta}\phi_3 \|_{L^2(\tau_0\times M)}^2\\
  \begin{aligned}
     &\leq \sum_{\tau\in\N_0,\;\xi\in\Z} \Bigg\|
          \sum_{\substack{(m_1,n_1,m_2,n_2,m_3,n_3)\in\sM:\\\tau=\lambda_{m_1,n_1}+\lambda_{m_2,n_2}+\lambda_{m_3,n_3},\\\xi=m_1+m_2+m_3}}
          \prod_{j=1}^3 \Pi_{n_j}\Theta_{m_j}\phi_j \Bigg\|_{L^2(\S^2)}^2\\
     &\leq \sum_{\tau\in\N_0,\;\xi\in\Z} \Biggl[
          \sum_{\substack{(m_1,n_1,m_2,n_2,m_3,n_3)\in\sM:\\\tau=\lambda_{m_1,n_1}+\lambda_{m_2,n_2}+\lambda_{m_3,n_3},\\\xi=m_1+m_2+m_3}}
          \biggl\| \prod_{j=1}^3\Pi_{n_j}\Theta_{m_j}\phi_j \biggr\|_{L^2(\S^2)} \Biggr]^2.
  \end{aligned}
\end{multline*}
   In contrast to \cite[Proposition~5.1]{BGT05b}, we do not estimate the number of terms of the inner sum,
   but we go back to the physical space: We set $a_{m_j,n_j}^{(j)} := \| \Pi_{n_j}\Theta_{m_j}\phi_j \|_{L^2(\S^2)}$
   for $j=1,2,3$ and apply Lemma~\ref{lem:trilin_spect_clust_est} as well as Plancherel's identity
   with respect to $t$ and $\theta$ to obtain
\begin{multline}\label{eq:trilin_first_step}
  \| P_{\sS}P_{N_1}e^{it\Delta}\phi_1P_{N_2}e^{it\Delta}\phi_2P_{N_3}e^{it\Delta}\phi_3 \|_{L^2(\tau_0\times M)}^2\\
  \begin{aligned}
     &\lesssim (N_2N_3)^{\frac12} \sum_{\tau\in\N_0,\;\xi\in\Z} \Biggl( 
          \sum_{\substack{(m_1,n_1,m_2,n_2,m_3,n_3)\in\sM:\\\tau=\lambda_{m_1,n_1}+\lambda_{m_2,n_2}+\lambda_{m_3,n_3},\\\xi=m_1+m_2+m_3}}
          \prod_{j=1}^3 a_{m_j,n_j}^{(j)} \Biggr)^2\\
     &\lesssim (N_2N_3)^{\frac12} \biggl\| \sum_{\sM}
          \prod_{j=1}^3e^{-i\lambda_{m_j,n_j}t}e^{im_j\theta} a_{m_j,n_j}^{(j)}
         \biggr\|_{L^2_{t,\theta}(\tau_0\times\S)}^2.
  \end{aligned}
\end{multline}
  Choose $p_1>4$ and $12<p_3<\infty$ and let $p_2>4$ be defined via the H\"older relation
  $\frac12=\frac{1}{p_1}+\frac{1}{p_2}+\frac{1}{p_3}$.
  We apply H\"older's estimate to obtain
\begin{multline*}
  \| P_{\sS}P_{N_1}e^{it\Delta}\phi_1P_{N_2}e^{it\Delta}\phi_2P_{N_3}e^{it\Delta}\phi_3 \|_{L^2(\tau_0\times M)}\\
    \lesssim (N_2N_3)^{\frac14} \biggl\| \sum_{(m_1,n_1)\in\sS} e^{-i\lambda_{m_1,n_1}t}e^{im_1\theta} a_{m_1,n_1}^{(1)}
          \biggr\|_{L^{p_1}_{t,\theta}(\tau_0\times\S)}\\
    \times \prod_{j=2}^3\, \biggl\| \sum_{(m_j,n_j)\in\sN_j} e^{-i\lambda_{m_j,n_j}t}e^{im_j\theta} a_{m_j,n_j}^{(j)}
                 \biggr\|_{L^{p_j}_{t,\theta}(\tau_0\times\S)}. 
\end{multline*}
  We estimate the first term using Corollary~\ref{coro:lin_est2} and the other terms using
  Lemma~\ref{lem:lin_est}. Then, we obtain for all $\varepsilon>0$
\begin{multline*}
   \| P_{\sS}P_{N_1}e^{it\Delta}\phi_1P_{N_2}e^{it\Delta}\phi_2P_{N_3}e^{it\Delta}\phi_3 \|_{L^2(\tau_0\times M)} \\
   \begin{aligned}
     &\lesssim \bigl(N_2N_3\bigr)^{\frac14} M^{\frac12-\frac{2}{p_1}-\varepsilon}N_2^{\frac32-\frac{1}{p_1}-\frac{3}{p_2}+\varepsilon}
         N_3^{1-\frac{3}{p_3}} \prod_{j=1}^3\|\phi_j\|_{L^2(M)}\\
     &\lesssim \biggl(\frac{N_2}{N_1}+\frac{1}{N_2}\biggr)^{\frac12-\frac{2}{p_1}-\varepsilon} N_2^{\frac34+\frac{3}{p_3}}
         N_3^{\frac54-\frac{3}{p_3}} \prod_{j=1}^3\|\phi_j\|_{L^2(M)}.
   \end{aligned}
\end{multline*}
  Since $p_1>4$ and $p_3>12$, this implies the desired estimate provided
  $\varepsilon>0$ is sufficiently small.
\end{proof}

\begin{remark}
  The proof of Proposition~\ref{prop:trilin_est} does not extend to the case $\S\times\S^2_\rho$ directly,
  where $\S^2_\rho$ is the embedded sphere of radius $\rho>0$ in $\R^3$. However, preliminary calculations suggest
  that Proposition~\ref{prop:trilin_est} may be proved in the more general case by more
  technical arguments. This will be addressed in the PhD thesis of the second author.
\end{remark}

\section{Function spaces and the nonlinear estimate}\label{sect:funct}
We briefly recall the function spaces $U^p$ and $V^p$ introduced by Koch--Tataru \cite{KT05}, which have been successfully employed in the context of critical dispersive equations. We refer the reader to \cite{HHK09} or \cite{KTV14} for more details and to \cite[Section 2]{HTT11}, \cite[Section 2]{H13}, and \cite[Section 2]{HTT14} for this machinery in the context of the nonlinear Schr\"odinger equations on manifolds. 
\begin{definition}
Let $1\leq p<\infty$.
\begin{enumerate}
\item A step function $a\colon\R \to L^2$ is called a $U^p$-atom, if
\[
a(t)=\sum_{k=1}^K \chi_{[t_{k-1},t_k)}a_k, \quad \sum_{k=1}^K\|a_k\|_{L^2}^p=1
\]
for a partition $-\infty<t_0<\ldots<t_K\leq \infty$. The space $U^p$ is defined as the corresponding atomic space.
\item The space $V^p$ is the space of right-continuous functions $v\colon\R \to L^2$ such that
\[\|v\|_{V^p}^p
=\sup_{-\infty<t_0<\ldots<t_K\leq \infty}\sum_{k=1}^K\|v(t_k)-v(t_{k-1})\|_{L^2}^p<+\infty
\]
with the convention $v(+\infty):=0$, and in addition we require $\lim_{t \to -\infty}v(t)=0$.
\end{enumerate}
\end{definition}

We use the resolution spaces as defined in \cite[Definition~2.3]{H13}:
\begin{definition} Let $s\in \R$.
\begin{enumerate}
\item $X^s$ is defined as the space of all $u\colon\R \to H^s(M)$ such that $e^{-it\Delta} P_Nu \in U^2$ for all dyadic $N\geq 1$ and
\[
\|u\|_{X^s}:=\biggl(\sum_{N\geq 1}N^{2s}\|e^{-it\Delta} P_Nu\|_{U^2}^2\biggr)^{\frac12}<+\infty.
\]
\item $Y^s$ is defined as the space of all $u\colon\R \to H^s(M)$ such that $e^{-it\Delta} P_Nu \in V^2$ for all dyadic $N\geq 1$ and
\[
\|u\|_{Y^s}:=\biggl(\sum_{N\geq 1}N^{2s}\|e^{-it\Delta} P_Nu\|_{V^2}^2\biggr)^{\frac12}<+\infty.
\]
\item For an interval $\tau\subset \R$ we denote by $X^s(\tau)$ resp.\ $Y^s(\tau)$ the restriction space.
\end{enumerate}
\end{definition}

Next, we show how Proposition \ref{prop:trilin_est} implies Theorem \ref{thm:main}. We remark that this derivation does not depend on the specifics of $M=\S\times \S^2$, it is similar to \cite[Corollary~3.7]{H13}, cp.\ also \cite{HTT11,HTT14} for corresponding arguments using unit scales instead of dyadic scales.
\begin{prop}\label{prop:trilin_est_y}
There exists $\delta>0$ such that for all dyadic numbers $N_1\geq N_2\geq N_3\geq 1$ and $P_{N_j}u_j\in Y^0$
($j=1,2,3$) the following holds true
\begin{equation}\label{eq:trilin_est_y}
   \|P_{N_1}\widetilde{u_1}P_{N_2}\widetilde{u_2}P_{N_3}\widetilde{u_3}\|_{L^2(\tau_0\times M)}
   \lesssim \biggl(\frac{N_3}{N_1}+\frac{1}{N_2}\biggr)^\delta N_2N_3 \prod_{j=1}^3 \|P_{N_j}u_j\|_{Y^0},
\end{equation}
where $\widetilde{u_j}$ denotes either $u_j$ or $\overline{u_j}$.
\end{prop}
\begin{proof}
Since the $L^2$-norm on the left hand side does not change under complex conjugation of any factor, we may ignore possible complex conjugations.

\emph{Step 1.} We start proving estimate \eqref{eq:trilin_est_y} with $Y^0$ replaced by $X^0$. In this case, it suffices to consider $U^2$-atoms $a_1$, $a_2$, $a_3$, given as
\[
P_{N_j}a_j=\sum_{k=1}^{K_j}\chi_{I_{k,j}}e^{it\Delta}P_{N_j}\phi_{k,j}, \quad \sum_{k=1}^{K_j}\|\phi_{k,j}\|^2_{L^2}=1,
\]
with pairwise disjoint right-open intervals $I_{1,j}, I_{2,j}, \ldots, I_{K_j,j}$.
Now,
\[
   \|P_{N_1}a_1P_{N_2}a_2P_{N_3}a_3\|_{L^2}^2\leq \sum_{k_1,k_2,k_3}\|e^{it\Delta}P_{N_1}\phi_{k_1,1}e^{it\Delta}P_{N_2}\phi_{k_2,2}e^{it\Delta}P_{N_3}\phi_{k_3,3}\|^2_{L^2}
\]
and Proposition~\ref{prop:trilin_est} implies
\[
   \|P_{N_1}a_1P_{N_2}a_2P_{N_3}a_3\|_{L^2}\leq C_{\delta}(N_1,N_2,N_3),
\]
with the constant $ C_{\delta}(N_1,N_2,N_3)$ from Proposition~\ref{prop:trilin_est}, which yields
\begin{equation}\label{eq:trilin_est_u2}
   \|P_{N_1}u_1P_{N_2}u_2P_{N_3}u_3\|_{L^2}
   \leq  C_{\delta}(N_1,N_2,N_3) \prod_{j=1}^3 \|e^{-it\Delta}P_{N_j}u_j\|_{U^2}.
\end{equation}

\emph{Step 2.} Now, choosing $N_1=N_2=N_3=N$ and $\phi_1=\phi_2=\phi_3$ in Proposition~\ref{prop:trilin_est}, we obtain
\begin{equation*}
  \|P_Ne^{it\Delta}\phi\|_{L^6}\ls N^{\frac23}\|P_N\phi\|_{L^2}.
\end{equation*}
As above, the estimate carries over to $U^6$-atoms, hence
\begin{equation*}
  \|P_Nu\|_{L^6}\ls N^{\frac23}\|e^{-it\Delta}P_Nu\|_{U^6},
\end{equation*}
and for general $N_1\geq N_2\geq N_3\geq 1$, by H\"older's inequality,
\begin{equation}\label{eq:trilin_est_u6}
\|P_{N_1}u_1P_{N_2}u_2P_{N_3}u_3\|_{L^2}
   \ls (N_1N_2N_3)^{\frac23} \prod_{j=1}^3 \|e^{-it\Delta}P_{N_j}u_j\|_{U^6}.
 \end{equation}
Also, by H\"older's inequality and the Sobolev embedding, see \cite[formula~(2.6)]{S88} and \cite[Lemma~3.4]{H13}, we obtain
\begin{align*}
  \|P_{N_1}u_1P_{N_2}u_2P_{N_3}u_3\|_{L^2}&\leq|\tau_0|^{\frac12}\|P_{N_1}u_1\|_{L^\infty_t L^2_x}\| P_{N_2}u_2\|_{L^\infty}\|P_{N_3}u_3\|_{L^\infty}\\
  &\ls  (N_2N_3)^{\frac32}\prod_{j=1}^3 \|P_{N_j}u_j\|_{L^\infty_t L^2_x}.
\end{align*}
For any $p\geq 1$, using $U^p\hookrightarrow L^\infty_tL^2_x$, we obtain the bound
\begin{equation}\label{eq:off}
  \|P_{N_1}u_1P_{N_2}u_2P_{N_3}u_3\|_{L^2}
  \ls (N_2N_3)^{\frac32}\prod_{j=1}^3 \|e^{-it\Delta}P_{N_j}u_j\|_{U^p},
\end{equation}
which is not scale invariant, but the constant does not depend on $N_1$.

\emph{Step 3.} We distinguish two cases:

\underline{Case a)} $N_2N_3>N_1$. In this case, we interpolate \eqref{eq:trilin_est_u2} and \eqref{eq:trilin_est_u6} using \cite[Lemma~2.4]{H13} and obtain
\[
   \|P_{N_1}u_1P_{N_2}u_2P_{N_3}u_3\|_{L^2} \ls A_{\delta}\prod_{j=1}^3 \|e^{-it\Delta}P_{N_j}u_j\|_{V^2},
\]
where
\begin{align*}
A_{\delta}&=C_{\delta}(N_1,N_2,N_3) \Bigl( \ln \frac{(N_1N_2N_3)^{\frac23}}{C_{\delta}(N_1,N_2,N_3)} +1\Bigr)^3\\
&\ls C_{\delta}(N_1,N_2,N_3)\Bigl( \ln \frac{N_1}{N_3} +1\Bigr)^3\ls C_{\delta'}(N_1,N_2,N_3)
\end{align*}
for any $\delta'<\delta$.

\underline{Case b)} $N_2N_3\leq N_1$. Now, we interpolate \eqref{eq:trilin_est_u2} and \eqref{eq:off} using \cite[Lemma~2.4]{H13} and obtain
\[
   \|P_{N_1}u_1P_{N_2}u_2P_{N_3}u_3\|_{L^2} \ls B_{\delta}\prod_{j=1}^3 \|e^{-it\Delta}P_{N_j}u_j\|_{V^2},
 \]
where
\begin{align*}
B_{\delta}&=C_{\delta}(N_1,N_2,N_3) \Bigl( \ln \frac{(N_2N_3)^{\frac32}}{C_{\delta}(N_1,N_2,N_3)} +1\Bigr)^3\\
&\ls C_{\delta}(N_1,N_2,N_3)( \ln N_2+1)^3\ls C_{\delta'}(N_1,N_2,N_3)
\end{align*}
for any $\delta'<\delta$, and the claim follows.
\end{proof}
In order to prove Theorem \ref{thm:main} we intend to solve the integral equation
\begin{equation}\label{eq:int}
u(t)=e^{it\Delta}u_0 \mp i \mathcal{I}(|u|^4u)(t), \qquad \mathcal{I}(f)(t):=\int_0^t e^{i(t-s)\Delta}f(s) \,ds,
\end{equation} 
for $u_0 \in H^1(M)$ by invoking the contraction mapping principle in a small closed ball in the space $X^1(\tau_0)\subset C(\tau_0,H^1(M))$.
For this, it suffices to provide the following estimate, cp.\ \cite[Proposition~4.2]{H13} and \cite[Proposition~4.1]{HTT11}:
\begin{prop}\label{prop:nonl}
For all $u,v\in X^1(\tau_0)$,
\begin{equation*}
\|\mathcal{I}(|u|^4u)-\mathcal{I}(|v|^4v)\|_{X^1(\tau_0)}\ls \bigl(\|u\|^4_{X^1(\tau_0)}+\|v\|^4_{X^1(\tau_0)}\bigr) \|u-v\|_{X^1(\tau_0)}.
\end{equation*}
\end{prop}
\begin{proof}[Proof (sketch)]
Due to the polynomial structure of the nonlinearity it suffices to prove an estimate for $\mathcal{I}(\prod_{j=1}^5\widetilde{u_j})$ where $\widetilde{u_j}$ denotes either $u_j$ or $\overline{u_j}$. This is treated exactly as in \cite[Proposition~4.2]{H13} (and \cite[Proposition~4.1]{HTT11}), where Proposition \ref{prop:trilin_est_y} is the replacement for \cite[Corollary~3.7]{H13}. Note that the contribution $\Sigma_2$ in \cite[pp.\ 1285--1287]{H13} is void in the case $M=\S\times\S^2$ (but \cite[Lemmas~3.3 and 3.4]{H13} hold true on any smooth compact Riemannian $3$-manifold $M$).
\end{proof}

To conclude the proof of Theorem \ref{thm:main} one can iterate the
local well-posedness to arbitrarily large time intervals $[0,T)$ by
using the conservation of the mass and the energy, see \cite[pp.\
344--347]{HTT11} for more details.

\section{On the necessity of the tri-linear estimate}\label{sect:nec}
As explained above, the tri-linear estimate in Proposition
\ref{prop:trilin_est} on an arbitrary compact boundary-less $3$-dimensional
Riemannian manifold $M$ is sufficient to conclude small data global
well-posedness in $H^1(M)$. The proof relies on the contraction
mapping principle, which implies that the flow map $F\colon u_0\mapsto u$
is smooth.

Conversely, we can show that the version of the  tri-linear estimate in Proposition~\ref{prop:trilin_est}
with $\delta=0$ is necessary for local
well-posedness with a smooth flow. We follow the argument of
\cite[Remark~2.12]{BGT05a}, which concerns bi-linear estimates in the context of the
cubic NLS.

Fix $T>0$ and consider the map
\[F\colon H^1(M) \to H^1(M), \quad F(u_0)=u(T),\]
where $u$ is a solution of \eqref{eq:nls} with initial data
$u(0)=u_0$. The fifth order differential of $F$ at the origin is
given by
\begin{multline*}
D^5F(0)(h_1,\ldots,h_5)\\
=\mp 12i\int_0^Te^{i(T-\tau)\Delta_g}\sum_{\sigma}
H_{\sigma(1)} (\tau)\overline{H_{\sigma(2)}}(\tau) H_{\sigma(3)}(\tau)
\overline{H_{\sigma(4)}}(\tau) H_{\sigma(5)}(\tau)\, d\tau,
\end{multline*}
where $H_j(\tau):=e^{i\tau\Delta_g}h_{j}$ and we sum over the $10=\binom{5}{2}$
of the $5!=120$
permutations $\sigma \in S_5$ which give rise to different pairs
$(\sigma(2),\sigma(4))$. Indeed, from \eqref{eq:int} it follows that
$DF(0)(h)=e^{iT\Delta_g}h$, $D^j F(0)=0$ for $2\leq j \leq 4$ and we
obtain the above formula.
If we specify to $h_2=h_3=h_4=h_5$ we obtain two
contributions
\[
\sum_{\sigma}
H_{\sigma(1)} \overline{H_{\sigma(2)}} H_{\sigma(3)}
\overline{H_{\sigma(4)}}H_{\sigma(5)} =6H_1|H_2|^4+4\overline{H_1}H_2^3\overline{H_2}.
\]
Now, let us assume that $D^5F(0)\colon \bigl(H^1(M)\bigr)^5\to H^1(M)$ is bounded. Then, we infer
\[\biggl|\int_M D^5F(0)(h_1,h_2,\ldots,h_2) \overline{H_1}(T)\,dx\biggr|\ls\|h_1\|_{H^1}\|h_1\|_{H^{-1}}\|h_2\|_{H^1}^4.\]
Because of
\[
\Re\{ 6|H_1|^2|H_2|^4+4\overline{H_1}^2H_2^3\overline{H_2}\}\geq 2|H_1|^2|H_2|^4,
\]
we conclude that
\[
\int_0^T\int_M |H_1|^2|H_2|^4 \,dx \,dt\ls \|h_1\|_{H^1}\|h_1\|_{H^{-1}}\|h_2\|_{H^1}^4.
\]
We set $h_1=P_{N_1}\phi_1$, and for $\phi_2,\phi_3\in H^1(M)$ we write
\[
e^{it\Delta_g}\phi_2 e^{it\Delta_g}\phi_3=\frac{1}{4}\bigl\{(e^{it\Delta_g}\phi_2+e^{it\Delta_g}\phi_3)^2-(e^{it\Delta_g}\phi_2-e^{it\Delta_g}\phi_3)^2\bigr\}\]
to obtain the
bound
\begin{multline*}
\|e^{it\Delta_g}P_{N_1}\phi_1 e^{it\Delta_g}\phi_2
e^{it\Delta_g}\phi_3 \|_{L^2([0,T]\times M)}\\
\ls \|P_{N_1}\phi_1\|_{L^2(M)}\|\phi_2\|_{H^1(M)}\|\phi_3\|_{H^1(M)},
\end{multline*}
which implies the estimate in Proposition \ref{prop:trilin_est}, but only with $\delta=0$.

\bibliography{literature}\bibliographystyle{amsplain}\label{sect:refs}

\end{document}